\documentclass[12pt,reqno]{amsart}
\usepackage[margin=1in]{geometry}
\usepackage{amssymb,amsfonts,amsmath,mathrsfs,etoolbox,mathtools,bm,bbm}
\usepackage[breaklinks=true,colorlinks=true,linkcolor=teal,citecolor=teal,filecolor=teal,urlcolor=black!50!blue]{hyperref}
\hypersetup{linktocpage}
\usepackage[hiresbb]{graphicx,xcolor}
\allowdisplaybreaks[4]

\newtheorem{theorem}{Theorem}[section]
\newtheorem{convention}[theorem]{Convention}
\newtheorem{corollary}[theorem]{Corollary}
\newtheorem{lemma}[theorem]{Lemma}
\newtheorem{proposition}[theorem]{Proposition}

\theoremstyle{definition}
\newtheorem{remark}[theorem]{Remark}

\numberwithin{equation}{section}

\renewcommand{\Im}{\mathrm{Im}}
\renewcommand{\Re}{\mathrm{Re}}

\patchcmd{\section}{\scshape}{\bfseries}{}{}
\makeatletter
\renewcommand{\@secnumfont}{\bfseries}
\makeatother
\makeatletter\newcommand{\tpmod}[1]{{\@displayfalse \pmod{#1}}}

\begin{document}

\title{Mixed Moments of the Riemann Zeta and Dirichlet $L$-Functions}

\author{Ikuya Kaneko}
\address{The Division of Physics, Mathematics and Astronomy, California Institute of Technology, 1200 E. California Blvd., Pasadena, CA 91125, USA}
\urladdr{\href{https://sites.google.com/view/ikuyakaneko/}{https://sites.google.com/view/ikuyakaneko/}}
\email{ikuyak@icloud.com}

\thanks{The author acknowledges the support of the Masason Foundation.}

\subjclass[2010]{11M06 (primary); 11F03, 11F72 (secondary)}

\keywords{Motohashi's formula, Riemann zeta function, Dirichlet $L$-functions}

\date{\today}

\dedicatory{To Yoichi Motohashi with admiration}

\begin{abstract}
We prove Motohashi's formula for a mixed second moment of the Riemann~zeta function~and a Dirichlet $L$-function attached to a primitive Dirichlet character modulo $q \in \mathbb{N}$. If $q$ is an odd prime, our reciprocity formula is consistent with Motohashi's result in the early 1990s. The cubic moment~side~features two versions of central $L$-values of automorphic forms on $\Gamma_{0}(q) \backslash \mathbb{H}$. The methods involve a blend of analytic number theory and automorphic~forms.
\end{abstract}

\maketitle
\tableofcontents

\section{Introduction}
In a series of work culminating in~\cite{Motohashi1993}, Motohashi proves a dualising formula relating the fourth moment of the Riemann zeta function to the cubic moment of $L$-functions attached to automorphic forms (Maa{\ss} and holomorphic forms as well as Eisenstein series) for $\mathrm{SL}_{2}(\mathbb{Z})$. Given a primitive Dirichlet character $\psi$ modulo $q \in \mathbb{N}$, we address Motohashi's formula for
\begin{equation*}
\int_{\mathbb{R}} \left|\zeta \left(\frac{1}{2}+it \right) L \left(\frac{1}{2}+it, \psi \right) \right|^{2} g(t) dt.
\end{equation*}
This problem was posed by Motohashi~\cite{Motohashi1993-2} at the RIMS conference in 1992. He handles the case where $q$ is an odd prime, and then considers the case of $q$ squarefree. Our identity~is valid for all $q \in \mathbb{N}$, and describes all terms on the cubic moment side in terms of newforms. A back-of-the-envelope argument in Section~\ref{sketch} implies that we should expect something like
\begin{equation}\label{rough}
\int_{\mathbb{R}} \left|\zeta \left(\frac{1}{2}+it \right) L \left(\frac{1}{2}+it, \psi \right) \right|^{2} g(t) dt 
\leftrightsquigarrow \sum_{f \in \mathcal{B}^{\ast}(q)} L \left(\frac{1}{2}, f \right)^{2} L \left(\frac{1}{2}, f \otimes \overline{\psi} \right) \check{g}(t_{f}),
\end{equation}
where $\mathcal{B}^{\ast}(q)$ stands for an orthonormal basis of Hecke--Maa{\ss} newforms of level $q$ and trivial central character, $t_{f} \in \mathbb{R} \cup [-i\vartheta, i\vartheta]$ is the spectral parameter of $f$ with $\vartheta = \frac{7}{64}$ at the current state of knowledge, and $\check{g}$ is a certain elaborate integral transform of $g$.

\subsection{Statement of Results}
Fix a primitive Dirichlet character $\psi$ modulo $q$, and let $\mathcal{R}_{4}^{+}$~be the subdomain of $\mathbb{C}^{4}$ where all four parameters have real parts greater than one. This work aims to generalise a spectral reciprocity formula of Motohashi~\cite{Motohashi1997-2}. Given a sufficiently well-behaved test function $g$ (see Convention~\ref{convention}), we consider the \textit{mixed second moment}
\begin{equation}\label{second-moment}
\mathcal{M}_{2}(s_{1}, s_{2}, s_{3}, s_{4}; g; \psi)
\coloneqq \int_{\mathbb{R}} \zeta(s_{1}+it) L(s_{2}+it, \psi) \zeta(s_{3}-it) L(s_{4}-it, \overline{\psi}) g(t) dt,
\end{equation}
where $(s_{1}, s_{2}, s_{3}, s_{4}) \in \mathcal{R}_{4}^{+}$. We initially work in the region of absolute convergence, and then take the limit $(s_{1}, s_{2}, s_{3}, s_{4}) \to (\frac{1}{2}, \frac{1}{2}, \frac{1}{2}, \frac{1}{2})$ after the meromorphic continuation.

The cubic moment side entails twisted spectral mean values
\begin{multline*}
\mathcal{J}_{\pm}^{\text{Maa{\ss}}} \coloneqq \sum_{q_{0} M \mid q} 
\sum_{f \in \mathcal{B}^{\ast}(q_{0})} \frac{L(\frac{s_{1}+s_{2}+s_{3}+s_{4}-1}{2}, f) 
L(\frac{1+s_{1}-s_{2}+s_{3}-s_{4}}{2}, f) L(\frac{1+s_{1}-s_{2}-s_{3}+s_{4}}{2}, f \otimes \overline{\psi})}
{L(1, \operatorname{ad} f)}\\
\times \Theta_{q}^{\text{Maa{\ss}}}(s_{1}, s_{2}, s_{3}, s_{4}, f),
\end{multline*}
\begin{multline*}
\mathcal{J}_{\pm}^{\text{Eis}} \coloneqq \sum_{c_{\chi}^{2} \mid M \mid q} 
\int_{\mathbb{R}} \frac{L(\frac{1+s_{1}-s_{2}+s_{3}-s_{4}}{2}+it, \chi) 
L^{\frac{q}{M}}(\frac{s_{1}+s_{2}+s_{3}+s_{4}-1}{2}+it, \chi) 
L(\frac{1+s_{1}-s_{2}-s_{3}+s_{4}}{2}+it, \chi \overline{\psi})}{L^{q}(1+2it, \chi^{2})}\\
\times \frac{L(\frac{s_{1}+s_{2}+s_{3}+s_{4}-1}{2}-it, \overline{\chi}) 
L^{\frac{q}{M}}(\frac{1+s_{1}-s_{2}+s_{3}-s_{4}}{2}-it, \overline{\chi}) 
L(\frac{1+s_{1}-s_{2}-s_{3}+s_{4}}{2}-it, \overline{\chi \psi})}{L^{q}(1-2it, \overline{\chi}^{2})}\\
\times \Theta_{\chi, q}^{\text{Eis}}(s_{1}, s_{2}, s_{3}, s_{4}, t) \frac{dt}{2\pi},
\end{multline*}
\begin{multline*}
\mathcal{J}_{+}^{\text{hol}} \coloneqq  \sum_{q_{0} M \mid q} 
\sum_{f \in \mathcal{B}_{\text{hol}}^{\ast}(q_{0})} \frac{L(\frac{s_{1}+s_{2}+s_{3}+s_{4}-1}{2}, f) 
L(\frac{1+s_{1}-s_{2}+s_{3}-s_{4}}{2}, f) L(\frac{1+s_{1}-s_{2}-s_{3}+s_{4}}{2}, f \otimes \overline{\psi})}
{L(1, \operatorname{ad} f)}\\
\times \Theta_{q}^{\text{hol}}(s_{1}, s_{2}, s_{3}, s_{4}, f),
\end{multline*}
where $\Theta_{q}^{\Diamond}(s_{1}, s_{2}, s_{3}, s_{4}, f)$ for $\Diamond \in \{\text{Maa{\ss}}, \text{hol} \}$ and $\Theta_{\chi, q}^{\text{Eis}}(s_{1}, s_{2}, s_{3}, s_{4}, t)$ are completely explicit expressions defined in~\eqref{Theta-1}--\eqref{Theta-3}, with the dependence on $\psi$ and $M$ suppressed. Moreover, the notation $L^{q}$ denotes an $L$-function with Euler factors dividing $q$ removed. The definition of the adjoint square $L$-function $L(s, \operatorname{ad} f)$ in $\Re(s) > 1$ is given in~\eqref{adjoint}. Write
\begin{equation*}
\mathcal{J}_{\pm} \coloneqq \mathcal{J}_{\pm}^{\text{Maa{\ss}}}+\mathcal{J}_{\pm}^{\text{Eis}}+\delta_{\pm = +} \mathcal{J}_{+}^{\text{hol}}.
\end{equation*}
We are now ready to state our spectral reciprocity formula to which we have alluded.
\begin{theorem}\label{main}
Let $(s_{1}, s_{2}, s_{3}, s_{4}) \in \mathbb{C}^{4}$, and let $\psi$ be a primitive Dirichlet character modulo $q \in \mathbb{N}$. If a test function $g$ obeys Convention~\ref{convention}, then we have that
\begin{multline*}
\mathcal{M}_{2}(s_{1}, s_{2}, s_{3}, s_{4}; g; \psi) = \mathcal{N}(s_{1}, s_{2}, s_{3}, s_{4}; g; \psi)\\
 + \sum_{\pm} \left(\mathcal{J}_{\pm}(s_{1}, s_{2} s_{3}, s_{4}; g; \psi)
 + \overline{\mathcal{J}_{\pm}(\overline{s_{3}}, \overline{s_{4}}, \overline{s_{2}}, \overline{s_{1}}; g; \psi)} \right),
\end{multline*}
where $\mathcal{N}$ denotes an explicitly computable main term.
\end{theorem}

When $q$ is an odd prime, Theorem~\ref{main} matches Motohashi's result~\cite{Motohashi1993-2}. In regard~to this restriction, he mentioned that ``this is to avoid unnecessary complexity, and in fact its elimination is by no means difficult." An absence of a description of the Fourier expansion of Eisenstein series for arbitrary conductors caused a problem at the time, and the Kuznetsov formula for nonsquarefree conductors was recently shown by Blomer--Khan~\cite{BlomerKhan2019} with the aid of explicit formul{\ae} for the Fourier coefficients of a complete orthonormal basis.

\begin{remark}
Let $\zeta_{K}(s)$ be the Dedekind zeta function associated to a quadratic number field~$K$ with discriminant $D$. Then the factorisation
\begin{equation*}
\zeta_{K}(s) = \zeta(s) L(s, \psi_{D})
\end{equation*}
holds for a quadratic character $\psi_{D}$ modulo $|D|$. Theorem~\ref{main} can be seen as an expression for the second moment of Dedekind zeta functions in terms of the cubic moment of automorphic $L$-functions on $\mathrm{GL}_{2}$.
\end{remark}

\begin{corollary}\label{corollary}
Let $\psi$ be a quadratic character modulo $q \ll T^{1-\varepsilon}$. Then we have that
\begin{equation*}
\int_{T}^{T+q^{\frac{1}{3}} T^{\frac{2}{3}}} \left|\zeta \left(\frac{1}{2}+it \right) L \left(\frac{1}{2}+it, \psi \right) \right|^{2} dt \ll_{\varepsilon} q^{\frac{1}{3}+\varepsilon} T^{\frac{2}{3}+\varepsilon}.
\end{equation*}
\end{corollary}

Corollary~\ref{corollary} is a variant of the short interval fourth moment bound of Iwaniec~\cite{Iwaniec1980}:
\begin{equation*}
\int_{T}^{T+T^{\frac{2}{3}}} \left|\zeta \left(\frac{1}{2}+it \right) \right|^{4} dt \ll_{\varepsilon} T^{\frac{2}{3}+\varepsilon}.
\end{equation*}

\subsection{Method of Proof}\label{sketch}
We now offer a heuristic overview on the genesis of our reciprocity formula in Theorem~\ref{main}. This is a back-of-the-envelope type sketch geared towards experts. For the sake of argument, we use approximate tools, while our proof is based on a circumspect analysis of $L$-functions in the region of absolute convergence. We also elide all polar terms, correction factors, and the averaging over $t$ whose purpose is to ensure convergence.

The overall strategy is inspired by the work of Motohashi~\cite[Sections~4.3--4.7]{Motohashi1997-2}. The off-diagonal contribution after Atkinson's dissection looks like
\begin{equation*}
\sum_{n, m} \tau(n, \overline{\psi}) \tau(n+m, \psi),
\end{equation*}
where the twisted divisor function is defined by\footnote{This notation should not induce confusion with the Gau{\ss} sum~\eqref{eq:Gauss-sum}.}
\begin{equation*}
\tau(m, \psi) \coloneqq \sum_{d \mid m} \psi(d).
\end{equation*}
We make use of an approximate functional equation (see Lemma~\ref{approximate-functional-equation})
\begin{equation*}
\tau(m, \psi) = \frac{1+\psi(m)}{\tau(\psi)} \sum_{\ell = 1}^{\infty} \frac{1}{\ell} 
f \left(\frac{\ell}{\sqrt{m}} \right) \ \sideset{}{^{\ast}} \sum_{h \tpmod{\ell q}} \psi(h) e_{\ell q}(hm),
\end{equation*}
where $f(x)$ and its derivatives are of rapid decay. This is thought of as a simple alternative to the Duke--Friedlander--Iwaniec circle method. We now make use of Vorono\u{\i} summation~to the sum over $n$, or the functional equation for the twisted Estermann zeta function. We are then led to sums of Kloosterman sums of the shape
\begin{equation*}
\sum_{q \mid c} \frac{S(m, \pm n; c)}{c} h \left(\frac{4\pi \sqrt{mn}}{c} \right).
\end{equation*}
This is consistent with~\cite[Equation~(1.6)]{Motohashi1997-3}. We now apply the Kloosterman summation formula (Theorem~\ref{Kuznetsov-formula}) of level $q$ and trivial central character. Two $\mathrm{GL}_{2}$ $L$-functions and one twisted $\mathrm{GL}_{2}$ $L$-function emerge when one executes the sums over $m$ and~$n$, from which~\eqref{rough} follows. Note that Motohashi~\cite{Motohashi1993-2,Motohashi1997-3} only handled a certain specific test function.

\subsection*{Acknowledgements}
The author is indebted to Peter Humphries, Eren Mehmet K{\i}ral, and Yoichi Motohashi for helpful comments and encouragement.

\section{Standard Conventions}
The Vinogradov symbol $A \ll B$ or the big $O$~notation $A = \mathrm{O}(B)$~signifies~that $|A| \leq C|B|$ for an effectively computable constant $C > 0$. We abbreviate $e(\alpha) \coloneqq \exp(2\pi i\alpha)$ and~$e_{p}(\alpha) \coloneqq e(\frac{\alpha}{p})$ for $\alpha \in \mathbb{C}$ and $p \in \mathbb{R}$. The Kronecker symbol $\delta_{\mathrm{S}}$ detects $1$ if the statement S is true and $0$ otherwise. We denote the Gau{\ss} sum associated to a Dirichlet character $\psi$ modulo $q$ by
\begin{equation}\label{eq:Gauss-sum}
\tau(\psi, h) \coloneqq \sum_{a \tpmod q} \psi(a) e_{q}(ah).
\end{equation}
We write $\tau(\psi) \coloneqq \tau(\psi, 1)$. Other frequently occurring exponential sums are the Ramanujan sums and Kloosterman sums
\begin{equation*}
r_{q}(h) \coloneqq \sum_{d \mid (h, q)} \mu \left(\frac{q}{d} \right) d, \qquad 
S(m, n; q) \coloneqq \sideset{}{^{\ast}} \sum_{a \tpmod q} e_{q}(am+\overline{a} n),
\end{equation*}
where $\overline{a}$ denotes the multiplicative inverse modulo $q$, namely a solution to $a \overline{a} \equiv 1 \tpmod q$.

Furthermore, we assume that a fixed test function $g$ obeys the following assumptions.
\begin{convention}\label{convention}
The function $g$ is real valued on $\mathbb{R}$, and there exists a large constant~$A > 0$ such that it is holomorphic and $\ll (1+|t|)^{-A}$ on a horizontal strip $|\Im(t)| \leq A$. All implicit constants in the ensuing Vinogradov symbols may possibly depend on $A$ where applicable.
\end{convention}

\section{Arithmetic Preliminaries}
This section summarises fundamental properties of the twisted Estermann zeta function. The Ramanujan expansion for the twisted divisor function is also established.

\subsection{Estermann Zeta Function}
For $h, \ell, q \in \mathbb{N}$ with $(h, \ell q) = 1$ and $\Re(s) > 1$, we define
\begin{equation*}
D_{\psi} \left(s, \xi; \frac{h}{\ell q} \right) \coloneqq \sum_{n = 1}^{\infty} \sigma_{\xi}(n, \psi) e_{\ell q}(hn) n^{-s},
\end{equation*}
where $\psi$ is an arbitrary primitive Dirichlet character modulo $q$, and
\begin{equation}\label{twisted-divisor-function}
\sigma_{\xi}(n, \psi) \coloneqq \sum_{d \mid n} \psi(d) d^{\xi}.
\end{equation}
Analytic properties of the twisted Estermann zeta function follow from the identity
\begin{equation}\label{Hurwitz}
D_{\psi} \left(s, \xi; \frac{h}{\ell q} \right) = (\ell q)^{\xi-2s} \sum_{a, b \tpmod{\ell q}} 
\psi(b) e_{\ell q}(abh) \zeta \left(s, \frac{a}{\ell q} \right) \zeta \left(s-\xi, \frac{b}{\ell q} \right),
\end{equation}
where for $\alpha \in \mathbb{R}$ and $\Re(s) > 1$,
\begin{equation*}
\zeta(s, \alpha) \coloneqq \sum_{n+\alpha > 0} (n+\alpha)^{-s}
\end{equation*}
is the Hurwitz zeta function. It has a meromorphic continuation to $\mathbb{C}$ with a simple~pole at $s = 1$ of residue 1, and satisfies the functional equation
\begin{equation}\label{functional-equation-for-Hurwitz-zeta-function}
\zeta(s, \alpha) = \sum_{\pm} G^{\mp}(1-s) \zeta^{(\pm \alpha)}(1-s),
\end{equation}
where $G^{\pm}(s) \coloneqq (2\pi)^{-s} \Gamma(s) \exp(\pm \frac{i\pi s}{2})$ and $\zeta^{(\alpha)}(s) \coloneqq \sum_{n = 1}^{\infty} e(\alpha n) n^{-s}$. The formula~\eqref{functional-equation-for-Hurwitz-zeta-function}~for $\alpha \in \mathbb{Q}$ can be viewed as a restatement of the Poisson summation formula in residue classes. As a function of a single variable $s$, the twisted Estermann zeta function has a meromorphic continuation to $\mathbb{C}$ with a simple pole at $s = 1+\xi$ of residue
\begin{equation*}
\frac{\overline{\psi}(h) \tau(\psi) L(1+\xi, \overline{\psi})}{\ell^{1+\xi} q},
\end{equation*}
provided $\xi \ne 0$. This follows from routine computations. In fact, one evaluates
\begin{align*}
\mathop{\mathrm{res}}_{s = 1+\xi} D_{\psi} \left(s, \xi; \frac{h}{\ell q} \right)
& = (\ell q)^{-2-\xi} \sum_{a, b \tpmod{\ell q}} \psi(b) e_{\ell q}(abh) \zeta \left(1+\xi, \frac{a}{\ell q} \right)\\
& = (\ell q)^{-2-\xi} \sum_{a \tpmod{\ell q}} \sum_{b_{1} = 1}^{q} \psi(b_{1}) 
\sum_{b_{2} = 0}^{\ell-1} e_{\ell q}(ab_{1} h) e_{\ell}(ab_{2} h) \zeta \left(1+\xi, \frac{a}{\ell q} \right).
\end{align*}
By orthogonality and the definition of the Gau{\ss} sum, the right-hand side is equal to
\begin{equation*}
\ell^{-1-\xi} q^{-2-\xi} \sum_{a, b \tpmod q} \psi(b) e_{q}(abh) \zeta \left(1+\xi, \frac{a}{q} \right)
 = \frac{\overline{\psi}(h) \tau(\psi) L(1+\xi, \overline{\psi})}{\ell^{1+\xi} q}.
\end{equation*}

Furthermore, the functional equation for $D_{\psi}(s, \xi; \frac{h}{\ell q})$ reads as follows.
\begin{theorem}\label{functional-equation-1}
The twisted Estermann zeta function enjoys the functional equation
\begin{align*}
\begin{split}
D_{\psi} \left(s, \xi; \frac{h}{\ell q} \right) &= \psi(-\overline{h})(2\pi)^{2s-\xi-2}(\ell q)^{\xi-2s+1} \Gamma(1-s) \Gamma(1+\xi-s)\\ 
&\quad \times \left[\left\{e \left(\frac{\xi}{4} \right)+\psi(-1) e \left(-\frac{\xi}{4} \right) \right\} D_{\psi} \left(1+\xi-s, \xi; \frac{\overline{h}}{\ell q} \right) \right.\\
&\quad - \left. \left\{e \left(\frac{2s-\xi}{4} \right)+\psi(-1) e \left(-\frac{2s-\xi}{4} \right) \right\} D_{\psi} \left(1+\xi-s; \xi; -\frac{\overline{h}}{\ell q} \right) \right],
\end{split}
\end{align*}
where $\overline{h}$ denotes the multiplicative inverse of $h$ modulo $\ell q$, namely $h \overline{h} \equiv 1 \tpmod{\ell q}$.
\end{theorem}
This is essentially equivalent to the twisted $\mathrm{GL}_{2}$ Vorono\u{\i} summation formula.

\begin{proof}
It follows from~\eqref{Hurwitz} that $D_{\psi}(s, \xi; \frac{h}{\ell q})$ equals
\begin{align*}
&(2\pi)^{2s-\xi-2}(\ell q)^{\xi-2s} \Gamma(1-s) \Gamma(1+\xi-s)\\
& \times \left[\left\{\psi(-1) e \left(\frac{\xi}{4} \right)+e \left(-\frac{\xi}{4} \right) \right\} \sum_{a, b \tpmod{\ell q}} 
\psi(b) e_{\ell q}(-abh) \zeta^{(\frac{a}{\ell q})}(1-s) \zeta^{(\frac{b}{\ell q})}(1+\xi-s) \right.\\
& - \left. \left\{e \left(\frac{2s-\xi}{4} \right)+\psi(-1) e \left(-\frac{2s-\xi}{4} \right) \right\} \sum_{a, b \tpmod{\ell q}} 
\psi(b) e_{\ell q}(abh) \zeta^{(\frac{a}{\ell q})}(1-s) \zeta^{(\frac{b}{\ell q})}(1+\xi-s) \right].
\end{align*}
We then appeal to the identity
\begin{equation*}
\zeta^{(\frac{a}{\ell q})}(1-s) \zeta^{(\frac{b}{\ell q})}(1+\xi-s) = (\ell q)^{2s-\xi-2} 
\sum_{c, d \tpmod{\ell q}} e_{\ell q}(ac+bd) \zeta \left(1-s, \frac{c}{\ell q} \right) \zeta \left(1+\xi-s, \frac{d}{\ell q} \right),
\end{equation*}
which follows from expanding $\zeta^{(\frac{a}{\ell q})}(1-s)$ and $\zeta^{(\frac{b}{\ell q})}(1+\xi-s)$ into sums over residue classes modulo $\ell q$. Then we observe
\begin{equation*}
\sum_{a, b \tpmod{\ell q}} \psi(b) e_{\ell q}(ac+bd-abh)
 = \sum_{b \tpmod{\ell q}} \psi(b) e_{\ell q}(bd) \delta_{c \equiv bh \tpmod{\ell q}}
 = \psi(c \overline{h}) e_{\ell q}(cd \overline{h}).
\end{equation*}
The proof of Theorem~\ref{functional-equation-1} is now complete.
\end{proof}

\subsection{Ramanujan Expansion}
The main result of this subsection is the following.
\begin{lemma}[Twisted Ramanujan expansion]\label{Ramanujan}
Let $\psi$ be a Dirichlet character modulo $q$.~Then we have for any $n \in \mathbb{N}$ and $\Re(\xi) > 0$ that
\begin{equation*}
\frac{\sigma_{\xi}(n, \psi)}{n^{\xi}} = \frac{L(1+\xi, \overline{\psi})}{\tau(\overline{\psi})} \sum_{\ell = 1}^{\infty} 
\frac{1}{\ell^{1+\xi}} \ \sideset{}{^{\ast}} \sum_{h \tpmod{\ell q}} \overline{\psi}(h) e_{\ell q}(hn).
\end{equation*}
\end{lemma}

\begin{proof}
We express the coprimality condition $(h, \ell q) = 1$ via M\"{o}bius inversion, justifying that the coefficients of $\ell^{-\xi}$ are equal on both sides of
\begin{equation}\label{Mobius-function}
\frac{\sigma_{\xi}(n, \psi) \tau(\overline{\psi})}{L(1+\xi, \overline{\psi}) n^{\xi}} = \sum_{\ell = 1}^{\infty} \frac{1}{\ell^{1+\xi}} 
\sum_{d \mid \ell} \mu(d) \overline{\psi}(d) \sum_{h \tpmod{\frac{\ell q}{d}}} \overline{\psi}(h) e_{\frac{\ell q}{d}}(hn).
\end{equation}
It is convenient for the reader to describe a straightforward way to prove Lemma~\ref{Ramanujan}. If we assume that $\psi$ is primitive,~then the right-hand side of~\eqref{Mobius-function} is equal to
\begin{equation*}
\sum_{\ell = 1}^{\infty} \frac{1}{\ell^{1+\xi}} \sum_{d = 1}^{\infty} \mu(d) 
\frac{\overline{\psi}(d)}{d^{1+\xi}} \sum_{h \tpmod{\ell q}} \overline{\psi}(h) e_{\ell q}(hn)
 = \frac{1}{L(1+\xi, \overline{\psi})} \sum_{\ell = 1}^{\infty} \ell^{-1-\xi} 
\sum_{h \tpmod{\ell q}} \overline{\psi}(h) e_{\ell q}(hn).
\end{equation*}
The identity $\tau(\psi, m) = \overline{\psi}(m) \tau(\psi)$ implies
\begin{equation*}
\frac{1}{L(1+\xi, \overline{\psi})} \sum_{\ell = 1}^{\infty} \ell^{-1-\xi} 
\sum_{h_{1} = 1}^{q} \overline{\psi}(h_{1}) e_{\ell q}(h_{1} n) \sum_{h_{2} = 0}^{\ell-1} e_{\ell}(h_{2} n)
 = \frac{\psi(n) \sigma_{-\xi}(n, \overline{\psi}) \tau(\overline{\psi})}{L(1+\xi, \overline{\psi})},
\end{equation*}
Lemma~\ref{Ramanujan} now follows from the relation $\sigma_{\xi}(n, \psi) n^{-\xi} = \psi(n) \sigma_{-\xi}(n, \overline{\psi})$.
\end{proof}

We state the following approximate functional equation since it is of independent interest.
\begin{lemma}\label{approximate-functional-equation}
Let $G(s)$ be a fixed even entire holomorphic function of rapid decay in vertical strips, satisfying $G(0) = 1$. Then we have for any $\xi \in \mathbb{C}$ that
\begin{multline*}
\sigma_{\xi}(n, \psi) = \frac{\psi(n)}{\tau(\psi)} \sum_{\ell = 1}^{\infty} \frac{1}{\ell^{1-\xi}} \ 
\sideset{}{^{\ast}} \sum_{h \tpmod{\ell q}} \psi(h) e_{\ell q}(hn) f_{\xi} \left(\frac{\ell}{\sqrt{n}} \right)\\
 + \frac{n^{\xi}}{\tau(\psi)} \sum_{\ell = 1}^{\infty} \frac{1}{\ell^{1+\xi}} \ \sideset{}{^{\ast}} 
\sum_{h \tpmod{\ell q}} \psi(h) e_{\ell q}(hn) f_{-\xi} \left(\frac{\ell}{\sqrt{n}} \right),
\end{multline*}
where for $a > |\Re(\xi)|$,
\begin{equation*}
f_{\xi}(x) \coloneqq \int_{(a)} x^{-w} L(1-\xi+w, \psi) \frac{G(w)}{w} \frac{dw}{2\pi i}.
\end{equation*}
\end{lemma}

\begin{proof}
Shifting the contour, we derive
\begin{equation*}
\int_{(a)} \sigma_{\xi-w}(n, \psi) n^{\frac{w}{2}} \frac{G(w)}{w} \frac{dw}{2\pi i} 
 = \sigma_{\xi}(n, \psi)+\int_{(-a)} \sigma_{\xi-w}(n, \psi) n^{\frac{w}{2}} \frac{G(w)}{w} \frac{dw}{2\pi i}.
\end{equation*}
Making a change of variables $w \mapsto -w$ gives
\begin{equation*}
\sigma_{\xi}(n, \psi) = \int_{(a)} \sigma_{\xi-w}(n, \psi) n^{\frac{w}{2}} \frac{G(w)}{w} \frac{dw}{2\pi i}
 + \psi(n) n^{\xi} \int_{(a)} \sigma_{-\xi-w}(n, \overline{\psi}) n^{\frac{w}{2}} \frac{G(w)}{w} \frac{dw}{2\pi i}.
\end{equation*}
By Lemma~\ref{Ramanujan}, this matches the desired claim.
\end{proof}

\subsection{Hecke Relation}
We record the Hecke relation for the twisted divisor function $\sigma_{\xi}(n, \psi)$.
\begin{lemma}\label{Hecke-multiplicativity}
The twisted divisor function~\eqref{twisted-divisor-function} enjoys the Hecke relation
\begin{equation*}
\sigma_{\xi}(mn, \psi) = \sum_{c \mid (m, n)} \mu(c) \psi(c) c^{\xi} \sigma_{\xi} \left(\frac{m}{c}, \psi \right) \sigma_{\xi} \left(\frac{n}{c}, \psi \right).
\end{equation*}
\end{lemma}

\begin{proof}
If $m = p_{1}^{k_{1}} p_{2}^{k_{2}} \cdots p_{r}^{k_{r}}$, $n = p_{1}^{\ell_{1}} p_{2}^{\ell_{2}} \cdots p_{r}^{\ell_{r}}$, and the claim holds for prime powers, then
\begin{align}\label{prime-power}
\begin{split}
\sigma_{\xi}(mn, \psi) &= \sigma_{\xi}(p_{1}^{k_{1}+\ell_{1}}, \psi) 
\sigma_{\xi}(p_{2}^{k_{2}+\ell_{2}}, \psi) \cdots \sigma_{\xi}(p_{r}^{k_{r}+\ell_{r}}, \psi)\\
& = \prod_{i = 1}^{r} \sum_{j = 0}^{\min(k_{i}, \ell_{i})} \mu(p_{i}^{j}) \psi(p_{i}^{j})
p_{i}^{j \xi} \sigma_{\xi}(p_{i}^{k_{i}-j}) \sigma_{\xi}(p_{i}^{\ell_{i}-j})\\
& = \sum_{c \mid (m, n)} \mu(c) \psi(c) c^{w} \sigma_{\xi} \left(\frac{m}{c} \right) \sigma_{\xi} \left(\frac{n}{c} \right).
\end{split}
\end{align}
By multiplicativity, we assume that $m = p^{k}$ and $n = p^{\ell}$. The claim follows if $\min(k, \ell) = 0$, thus we focus on the case where both are at least $1$. If $X = \psi(p) p^{\xi}$, then the right-hand side of~\eqref{prime-power} is equal to
\begin{align*}
&\sigma_{\xi}(p^{k}, \psi) \sigma_{\xi}(p^{\ell}, \psi)-\psi(p) p^{\xi} \sigma_{\xi}(p^{k-1}, \psi) \sigma_{\xi}(p^{\ell-1}, \psi)\\
& \quad = (1+X+\cdots+X^{k})(1+X+\cdots+X^{\ell})-X(1+X+\cdots+X^{k-1})(1+X+\cdots+X^{\ell-1})\\
& \quad = \frac{X^{k+\ell+1}-1}{X-1} = (1+\psi(p) p^{\xi}+\psi(p)^{2} p^{2\xi}+\cdots+\psi(p)^{k+\ell} p^{(k+\ell)\xi})
 = \sigma_{\xi}(p^{k+\ell}, \psi),
\end{align*}
as required. The proof of Lemma~\ref{Hecke-multiplicativity} is now complete.
\end{proof}

\section{The Kuznetsov Formula}
We formulate an asymmetric trace formula relating sums of Kloosterman sums to~Fourier coefficients of automorphic forms in the sense that the spectral and geometric sides possesses a quite different shape. We begin with an explanation of the Fourier expansions of Maa{\ss}~and holomorphic forms as well as Eisenstein series in the same manner as in~\cite[Section 3]{BlomerKhan2019}, and refer to this source for further details. Let $N \in \mathbb{N}$, and let
\begin{equation*}
\nu(N) \coloneqq \prod_{p \mid N} \left(1+\frac{1}{p} \right).
\end{equation*}
We equip $\Gamma_{0}(N) \backslash \mathbb{H}$ with the Petersson inner product
\begin{equation*}
\langle f, g \rangle \coloneqq \int_{\Gamma_{0}(N) \backslash \mathbb{H}} f(z) \overline{g(z)} \frac{dx dy}{y^{2}}.
\end{equation*}

The cuspidal spectrum is parametrised by pairs of $\Gamma_{0}(N)$-normalised newforms $f$ of level $N_{0} \mid N$ as well as integers $M \mid \frac{N}{N_{0}}$. This comes from orthonormalising $\{z \mapsto f(Mz): M \mid \frac{N}{N_{0}} \}$ by Gram--Schmidt. If $f$ is Maa{\ss}, then we write the Fourier expansion of the pair $(f, M)$ as
\begin{equation*}
\sqrt{2 \cosh(\pi t) y} \sum_{n \ne 0} \rho_{f, M, N}(n) K_{it}(2\pi |n|y) e(nx),
\end{equation*}
where
\begin{equation}\label{Fourier-coefficients}
\rho_{f, M, N}(n) = \frac{1}{L(1, \operatorname{ad} f)^{\frac{1}{2}}(MN \nu(N))^{\frac{1}{2}}} \prod_{p \mid N_{0}} 
\left(1-\frac{1}{p^{2}} \right)^{\frac{1}{2}} \sum_{d \mid (M, n)} d \xi_{f}(M, d) \lambda_{f} \left(\frac{n}{d} \right)
\end{equation}
for $n \in \mathbb{N}$ with the convention $\lambda_{f}(n) = 0$ for $n \not \in \mathbb{Z}$. Here $\xi_{f}(M, d)$ is a certain multiplicative function defined in~\cite[Section 3.2]{BlomerKhan2019} obeying
\begin{equation}\label{xi}
\xi_{f}(M, d) \ll_{\varepsilon} M^{\varepsilon} \left(\frac{M}{d} \right)^{\theta}.
\end{equation}
Furthermore, we define the adjoint square $L$-function by
\begin{equation}\label{adjoint}
L(s, \operatorname{ad} f) \coloneqq \zeta^{(N)}(2s) \sum_{n = 1}^{\infty} \frac{\lambda_{f}(n^{2})}{n^{s}},
\end{equation}
which may differ from the corresponding Langlands $L$-function by finitely many Euler factors. For $-n \in \mathbb{N}$, we have $\rho_{f, M, N}(n) = \varepsilon_{f} \rho_{f, M, N}(-n)$ if $f$ is Maa{\ss} of parity $\varepsilon_{f} \in \{\pm 1 \}$.

If $f$ is holomorphic of weight $k_{f}$ and level $N_{0} \mid N$, then we write the Fourier expansion of the pair $(f, M)$ as
\begin{equation*}
\left(\frac{2\pi^{2}}{\Gamma(k)} \right)^{\frac{1}{2}} \sum_{n \geq 1} \rho_{f, M, N}(n)(4\pi n)^{\frac{k_{f}-1}{2}} e(nz).
\end{equation*}
Then~\eqref{Fourier-coefficients} remains true for $n \in \mathbb{N}$, and we define $\rho_{f, M, N}(n) = 0$ for $-n \in \mathbb{N}$.

Unitary Eisenstein series for $\Gamma_{0}(N)$ are parametrised by a continuous parameter $s = \frac{1}{2}+it$, $t \in \mathbb{R}$, and pairs $(\chi, M)$, where $\chi$ is a primitive Dirichlet character of conductor $c_{\chi}$ and $M \in \mathbb{N}$ satisfies $c_{\chi}^{2} \mid M \mid N$. We define
\begin{equation*}
\widetilde{\mathfrak{n}}_{N}(M) \coloneqq \left(\prod_{\substack{p \mid N \\ p \nmid (M, \frac{N}{M})}} 
\frac{p}{p+1} \prod_{p \mid (M, \frac{N}{M})} \frac{p-1}{p+1} \right)^{\frac{1}{2}},
\end{equation*}
and write
\begin{equation*}
M = c_{\chi} M_{1} M_{2}, \qquad (M_{2}, c_{\chi}) = 1, \qquad M_{1} \mid c_{\chi}^{\infty}.
\end{equation*}
The Fourier coefficients of unitary Eisenstein series associated to the data $(N, M, t, \chi)$ are
\begin{equation*}
\rho_{\chi, M, N}(n, t) = \frac{C(\chi, M, t)|n|^{it}}{(N \nu(N))^{\frac{1}{2}} \widetilde{\mathfrak{n}}_{N}(M) L^{N}(1+2it, \chi^{2})} 
\left(\frac{M_{1}}{M_{2}} \right)^{\frac{1}{2}} 
\sum_{\delta \mid M_{2}} \delta \mu \left(\frac{M_{2}}{\delta} \right) \overline{\chi}(\delta) 
\sum_{\substack{cM_{1} \delta f = n \\ (c, \frac{N}{M}) = 1}} \frac{\chi(c) \overline{\chi}(f)}{c^{2it}}
\end{equation*}
for a constant $C(\chi, M, t)$ of modulus $1$.

\subsection{Kloosterman Summation Formula}
For $x > 0$, we introduce the integral kernels
\begin{align*}
\mathcal{J}^{+}(x, t) &\coloneqq \frac{\pi i}{\sinh(\pi t)}(J_{2it}(x)-J_{-2it}(x)),\\
\mathcal{J}^{-}(x, t) &\coloneqq \frac{\pi i}{\sinh(\pi t)}(I_{2it}(x)-I_{-2it}(x)) = 4\cosh(\pi t) K_{2it}(x),\\
\mathcal{J}^{\text{hol}}(x, k) &\coloneqq 2\pi i^{k} J_{k-1}(x) = \mathcal{J}^{+} \left(x, \frac{k-1}{2i} \right), \qquad k \in 2\mathbb{N}.
\end{align*}
If $H \in C^{3}((0, \infty))$ satisfies $x^{j} H^{(j)}(x) \ll \min(x, x^{-\frac{3}{2}})$ for $0 \leq j \leq 3$, then we define
\begin{equation*}
\mathscr{L}^{\Diamond} H \coloneqq \int_{0}^{\infty} \mathcal{J}^{\Diamond}(x, \cdot) H(x) \frac{dx}{x},
\end{equation*}
where $\Diamond \in \{+, -, \text{hol} \}$. Let
\begin{equation}\label{O}
\mathcal{O}_{q}(m, \pm n; H) \coloneqq \sum_{q \mid c} \frac{S(m, \pm n; c)}{c} H \left(\frac{4\pi \sqrt{mn}}{c} \right).
\end{equation}

\begin{theorem}[{Blomer et al.~\cite[Equation (2.14)]{BlomerHumphriesKhanMilinovich2020}}]\label{Kuznetsov-formula}
Let $x^{j} H^{(j)}(x) \ll \min(x, x^{-\frac{3}{2}})$ for $0 \leq j \leq 3$. Then we have for any $m, n, q \in \mathbb{N}$ that
\textup{
\begin{equation}\label{Kuznetsov}
\mathcal{O}_{q}(m, \pm n; H) = \mathcal{A}_{q}^{\text{Maa{\ss}}}(m, \pm n; \mathscr{L}^{\pm} H)
 + \mathcal{A}_{q}^{\text{Eis}}(m, \pm n; \mathscr{L}^{\pm} H)
 + \delta_{\pm = +} \mathcal{A}_{q}^{\text{hol}}(m, n; \mathscr{L}^{\text{hol}} H),
\end{equation}
}where
\textup{
\begin{align*}
\mathcal{A}_{q}^{\text{Maa{\ss}}}(m, n; h) &\coloneqq \sum_{q_{0} M \mid q} 
\sum_{f \in \mathcal{B}^{\ast}(q_{0})} 
\rho_{f, M, q}(m) \overline{\rho_{f, M, q}(n)} h(t_{f}),\\
\mathcal{A}_{q}^{\text{Eis}}(m, n; h) &\coloneqq \sum_{c_{\chi}^{2} \mid M \mid q} 
\int_{\mathbb{R}} \rho_{\chi, M, q}(m, t) \overline{\rho_{\chi, M, q}(n, t)} h(t) \frac{dt}{2\pi},\\
\mathcal{A}_{q}^{\text{hol}}(m, n; h) &\coloneqq \sum_{q_{0} M \mid q} 
\sum_{f \in \mathcal{B}_{\text{hol}}^{\ast}(q_{0})} 
\rho_{f, M, q}(m) \overline{\rho_{f, M, q}(n)} h(k_{f}).
\end{align*}
}
\end{theorem}
The above formulation mimics the Kuznetsov formula in the book of Motohashi~\cite{Motohashi1997-2}. Theorem~\ref{Kuznetsov-formula} can be thought of as a Kloosterman summation formula since it expresses sums of Kloosterman sums weighted by a fixed test function in terms of sums of automorphic~forms weighted by the corresponding integral transforms.

\section{Proof of Theorem~\ref{main}}
We embark on the proof of Theorem~\ref{main}. Recall Convention~\ref{convention} and the definition~\eqref{second-moment}.

\subsection{Initial Manipulations}
Let $\mathcal{R}_{4}^{+}$ (resp. $\mathcal{R}_{4}^{-}$) be the subdomain of $\mathbb{C}^{4}$ where all four parameters have real~parts greater than (resp. less than) one, namely
\begin{align*}
\mathcal{R}_{4}^{+} &\coloneqq \{(s_{1}, s_{2}, s_{3}, s_{4}) \in \mathbb{C}^{4}: \Re(s_{i}) > 1, \, 1 \leq i \leq 4 \},\\
\mathcal{R}_{4}^{-} &\coloneqq \{(s_{1}, s_{2}, s_{3}, s_{4}) \in \mathbb{C}^{4}: \Re(s_{i}) < 1, \, 1 \leq i \leq 4 \}.
\end{align*}
We write $\mathcal{M}_{2} \coloneqq \mathcal{M}_{2}(\texttt{s}; g; \psi) = \mathcal{M}_{2}(s_{1}, s_{2}, s_{3}, s_{4}; g; \psi)$ for a vector $\texttt{s} = (s_{1}, s_{2}, s_{3}, s_{4})$. Moving the contour appropriately, one sees that $\mathcal{M}_{2}$ is meromorphic over $\mathbb{C}$. Taking $(s_{1}, s_{2}, s_{3}, s_{4}) \in \mathcal{R}_{4}^{-} \cap \mathcal{B}_{4}$~and shifting the contour back to the original, we verify the meromorphic continuation to the domain $\mathcal{R}_{4}^{-} \cap \mathcal{B}_{4}$:
\begin{align*}
\begin{split}
\mathcal{M}_{2}
& = \int_{\mathbb{R}} \zeta(s_{1}+it) L(s_{2}+it, \psi) \zeta(s_{3}-it) L(s_{4}-it, \overline{\psi}) g(t) dt\\
& \quad + 2\pi \zeta(s_{1}+s_{3}-1) L(1-s_{1}+s_{2}, \psi) L(s_{1}+s_{4}-1, \overline{\psi}) g((s_{1}-1)i)\\
& \quad + 2\pi \zeta(s_{1}+s_{3}-1) L(s_{2}+s_{3}-1, \psi) L(1-s_{3}+s_{4}, \overline{\psi}) g((1-s_{3})i).
\end{split}
\end{align*}

To continue $\mathcal{M}_{2}$ from $\mathcal{R}_{4}^{+}$ to the vicinity of $\texttt{s} = (\frac{1}{2}, \frac{1}{2}, \frac{1}{2}, \frac{1}{2})$, we open up the Dirichlet series in the region of absolute convergence to recast
\begin{align*}
\mathcal{M}_{2} &= \sum_{n_{1}, n_{2}, n_{3}, n_{4} = 1}^{\infty} 
\frac{\psi(n_{2}) \overline{\psi(n_{4})}}{n_{1}^{s_{1}} n_{2}^{s_{2}} n_{3}^{s_{3}} n_{4}^{s_{4}}} 
\int_{\mathbb{R}} g(t) \left(\frac{n_{1} n_{2}}{n_{3} n_{4}} \right)^{-it} dt\\
& = \sum_{m, n = 1}^{\infty} \frac{\sigma_{s_{1}-s_{2}}(m, \psi) \sigma_{s_{3}-s_{4}}(n, \overline{\psi})}{m^{s_{1}} n^{s_{3}}} 
\hat{g} \left(\frac{1}{2\pi} \log \frac{m}{n} \right),
\end{align*}
where $\hat{g}$ denotes the Fourier transform
\begin{equation*}
\hat{g}(\xi) \coloneqq \int_{\mathbb{R}} g(t) e(-\xi t) dt.
\end{equation*}
We decompose
\begin{equation}\label{Atkinson-dissection}
\mathcal{M}_{2} = \mathcal{D}_{2}+\mathcal{OD}_{2}^{\dagger}+\mathcal{OD}_{2}^{\ddagger},
\end{equation}
where $\mathcal{D}_{2}$ is the diagonal contribution associated to $m = n$, while $\mathcal{OD}_{2}^{\dagger}$ (resp. $\mathcal{OD}_{2}^{\ddagger}$) is the off-diagonal contribution involving terms with $m > n$ (resp. $m < n$). The case where $m < n$ is symmetrical to the case where $m > n$, thus
\begin{equation*}
\mathcal{OD}_{2}^{\ddagger}(s_{1}, s_{2}, s_{3}, s_{4}; g; \psi)
 = \overline{\mathcal{OD}_{2}^{\dagger}(\overline{s_{3}}, \overline{s_{4}}, \overline{s_{1}}, \overline{s_{2}}; g; \psi)}.
\end{equation*}

\begin{lemma}\label{lem:Ramanujan}
Keep the notation as above. Then we have that
\begin{equation*}
\mathcal{D}_{2} = \hat{g}(0) \frac{\zeta(s_{1}+s_{3}) \zeta^{q}(s_{2}+s_{4}) 
L(s_{2}+s_{3}, \psi) L(s_{1}+s_{4}, \overline{\psi})}{\zeta^{q}(s_{1}+s_{2}+s_{3}+s_{4})}.
\end{equation*}
\end{lemma}

\begin{proof}
We follow~\cite{Motohashi2018} \textit{mutatis mutandis}. By definition,
\begin{equation*}
\mathcal{D}_{2} = \hat{g}(0) \sum_{m = 1}^{\infty} \frac{\sigma_{s_{1}-s_{2}}(m, \psi) \sigma_{s_{3}-s_{4}}(m, \overline{\psi})}{m^{s_{1}+s_{3}}},
\end{equation*}
and the Dirichlet coefficient is multiplicative in $m$. It is then convenient to write
\begin{equation*}
\sigma_{s_{1}-s_{2}}(m, \psi) \sigma_{s_{3}-s_{4}}(m, \overline{\psi})
 = \sum_{[u, v] \mid m} \psi(u) \overline{\psi}(v) u^{s_{1}-s_{2}} v^{s_{3}-s_{4}},
\end{equation*}
where $[u, v]$ denotes the least common multiple. Hence, the Dirichlet series in question equals
\begin{equation}\label{greatest-common-divisor}
\zeta(s_{1}+s_{3}) \sum_{u, v = 1}^{\infty} \psi(u) \overline{\psi}(v) \frac{u^{s_{1}-s_{2}} v^{s_{3}-s_{4}}}{[u, v]^{s_{1}+s_{3}}}
 = \zeta(s_{1}+s_{3}) \sum_{u, v = 1}^{\infty} \psi(u) \overline{\psi}(v) \frac{(u, v)^{s_{1}+s_{3}}}{u^{s_{2}+s_{3}} v^{s_{1}+s_{4}}},
\end{equation}
where $(u, v)$ denotes the greatest common divisor. Because M\"{o}bius inversion gives
\begin{equation*}
\sum_{f \mid g} \eta_{s}(f) = g^{s}, \qquad \eta_{s}(m) = \sum_{d \mid m} \mu(d) \left(\frac{m}{d} \right)^{s},
\end{equation*}
the right-hand side of~\eqref{greatest-common-divisor} becomes
\begin{multline*}
\zeta(s_{1}+s_{3}) \sum_{u, v = 1}^{\infty} \frac{\psi(u) \overline{\psi}(v)}{u^{s_{2}+s_{3}} v^{s_{1}+s_{4}}} \sum_{d \mid (u, v)} \eta_{s_{1}+s_{3}}(d)\\
 = \zeta(s_{1}+s_{3}) L(s_{2}+s_{3}, \psi) L(s_{1}+s_{4}, \overline{\psi}) \sum_{(d, q) = 1} \frac{\eta_{s_{1}+s_{3}}(d)}{d^{s_{1}+s_{2}+s_{3}+s_{4}}}.
\end{multline*}
The proof of Lemma~\ref{lem:Ramanujan} is now complete.
\end{proof}

\subsection{Off-Diagonal Contribution}
The analysis of the off-diagonal terms requires the full machinery of the spectral theory of automorphic forms associated to congruence subgroups. We handle the second term~in~\eqref{Atkinson-dissection}, and the contribution of the third term will be combined subsequently. Given $\alpha > 1$, let
\begin{equation*}
G(y, s) \coloneqq (1+y)^{-s} \hat{g} \left(\frac{1}{2\pi} \log(1+y) \right)
 = \int_{(\alpha)} \mathring{g}(\tau, s) y^{-\tau} \frac{d\tau}{2\pi i},
\end{equation*}
where $\mathring{g}$ denotes the Mellin transform
\begin{equation}\label{Mellin-transform}
\mathring{g}(\tau, s) \coloneqq \int_{0}^{\infty} y^{\tau-1}(1+y)^{-s} \hat{g} \left(\frac{1}{2\pi} \log(1+y) \right) dy
 = \Gamma(\tau) \int_{\mathbb{R}} \frac{\Gamma(s-\tau+it)}{\Gamma(s+it)} g(t) dt,
\end{equation}
provided $\Re(s) > \Re(\tau) > 0$.
\begin{lemma}\label{entire}
As a function of two complex variables, $\frac{\mathring{g}(\tau, s)}{\Gamma(\tau)}$ is entire in $\tau$ and $s$, and $\mathring{g}(\tau, s)$~is of rapid decay in $\tau$ as soon as $s$ and $\Re(\tau)$ are bounded.
\end{lemma}

\begin{proof}
Shifting the contour $\Im(t) = 0$ in~\eqref{Mellin-transform} downwards, we confirm the first assertion.~The second claim follows from an upward shift. See~\cite[Lemma 4.1]{Motohashi1997-2} for a similar statement.
\end{proof}

By definition,
\begin{equation*}
\mathcal{OD}_{2}^{\dagger} = \sum_{n, m = 1}^{\infty} \frac{\sigma_{s_{3}-s_{4}}(n, \overline{\psi}) 
\sigma_{s_{1}-s_{2}}(n+m, \psi)}{n^{s_{2}+s_{3}} (n+m)^{s_{1}-s_{2}}} G \left(\frac{m}{n}, s_{2} \right).
\end{equation*}
Lemma~\ref{Ramanujan} yields
\begin{equation*}
\mathcal{OD}_{2}^{\dagger} = \frac{L(1+s_{1}-s_{2}, \overline{\psi})}{\tau(\overline{\psi})} 
\sum_{\ell = 1}^{\infty} \ell^{s_{2}-s_{1}-1} \ \sideset{}{^{\ast}} \sum_{h \tpmod{\ell q}} \overline{\psi}(h) 
\sum_{n, m = 1}^{\infty} e_{\ell q}(h(n+m)) \frac{\sigma_{s_{3}-s_{4}}(n, \overline{\psi})}{n^{s_{2}+s_{3}}} G \left(\frac{m}{n}, s_{2} \right).
\end{equation*}
Given $\alpha > 1$, let
\begin{equation*}
\mathcal{R}_{4, \alpha} \coloneqq \{(s_{1}, s_{2}, s_{3}, s_{4}) \in \mathcal{R}_{4}^{+}: 
\Re(s_{1}) > \Re(s_{2})+1, \, \Re(s_{2}) > \alpha, \, \Re(s_{3}) > 1, \, \Re(s_{4}) > 1 \}.
\end{equation*}
In this domain, there holds
\begin{multline}\label{integral}
\mathcal{OD}_{2}^{\dagger} = \frac{L(1+s_{1}-s_{2}, \overline{\psi})}{\tau(\overline{\psi})} 
\sum_{\ell = 1}^{\infty} \ell^{s_{2}-s_{1}-1} \ \sideset{}{^{\ast}} \sum_{h \tpmod{\ell q}} 
\overline{\psi}(h)\\ \times \int_{(\alpha)} D_{\overline{\psi}} \left(s_{2}+s_{3}-\tau, s_{3}-s_{4}; \frac{h}{\ell q} \right) 
\zeta^{(\frac{h}{\ell q})}(\tau) \mathring{g}(\tau, s_{2}) \frac{d\tau}{2\pi i}.
\end{multline}
In anticipation of shifting the contour to the right, we introduce the new domain
\begin{equation*}
\mathcal{E}_{4, \beta} \coloneqq \{(s_{1}, s_{2}, s_{3}, s_{4}) \in \mathcal{R}_{4}^{+}: 
\Re(s_{1}+s_{3}) < \beta, \ \Re(s_{1}+s_{4}) < \beta, \ \Re(s_{1}+s_{2}+s_{3}+s_{4}) > 3\beta \},
\end{equation*}
where $\beta$ is sufficiently large. For the time being, we restrict our attention to $(s_{1}, s_{2}, s_{3}, s_{4}) \in \mathcal{R}_{4, \alpha} \cap \mathcal{E}_{4, \beta}$, which is not empty whenever $\beta > 1+\alpha$. Then the contour $(\alpha)$ in~\eqref{integral} is shifted to $(\beta)$ and the residue thereof is
\begin{multline*}
\frac{\mathring{g}(s_{2}+s_{4}-1, s_{2})}{q} L(1+s_{1}-s_{2}, \overline{\psi}) L(1+s_{3}-s_{4}, \psi) 
\sum_{(\ell, q) = 1} \ell^{-2-s_{1}+s_{2}-s_{3}+s_{4}} \sum_{n = 1}^{\infty} r_{\ell q}(n) n^{1-s_{2}-s_{4}}\\
 = \frac{\mathring{g}(s_{2}+s_{4}-1, s_{2})}{q} \frac{L(1+s_{1}-s_{2}, \overline{\psi}) 
L(1+s_{3}-s_{4}, \psi)}{\zeta^{q}(2+s_{1}-s_{2}+s_{3}-s_{4})}
\sum_{n = 1}^{\infty} \frac{\sigma_{1+s_{1}-s_{2}+s_{3}-s_{4}}(n) r_{q}(n)}{n^{s_{1}+s_{3}}},
\end{multline*}
where we used the coprimality condition between $\ell$ and $q$ as well as the standard Ramanujan expansion. The sum over $n$ evaluates to
\begin{multline*}
\sum_{n = 1}^{\infty} \frac{\sigma_{1+s_{1}-s_{2}+s_{3}-s_{4}}(n) r_{q}(n)}{n^{s_{1}+s_{3}}}
 = \zeta(s_{1}+s_{3}) \zeta(s_{2}+s_{4}-1) \sum_{c \mid q} \mu(c) c^{2-s_{1}-s_{2}-s_{3}-s_{4}}\\
\times \sum_{d \mid \frac{q}{c}} \mu \left(\frac{q}{cd} \right) d^{1-s_{1}-s_{3}} \sigma_{1+s_{1}-s_{2}+s_{3}-s_{4}}(d).
\end{multline*}
This gives rise to the polar term\footnote{It appears that Motohashi~\cite[Equation (1.5)]{Motohashi1997-3} is erroneous.}
\begin{equation}\label{P}
\mathcal{P} \coloneqq A_{q}(s_{1}, s_{2}, s_{3}, s_{4}) \mathring{g}(s_{2}+s_{4}-1, s_{2}) 
\frac{\zeta(s_{1}+s_{3}) \zeta(s_{2}+s_{4}-1) L(1+s_{1}-s_{2}, \overline{\psi}) 
L(1+s_{3}-s_{4}, \psi)}{\zeta^{q}(2+s_{1}-s_{2}+s_{3}-s_{4})},
\end{equation}
where $A_{q}$ is a multiplicative function in $q$ defined by
\begin{equation*}
A_{q}(s_{1}, s_{2}, s_{3}, s_{4}) \coloneqq \frac{1}{q} \sum_{c \mid q} \mu(c) c^{2-s_{1}-s_{2}-s_{3}-s_{4}} 
\sum_{d \mid \frac{q}{c}} \mu \left(\frac{q}{cd} \right) d^{1-s_{1}-s_{3}} \sigma_{1+s_{1}-s_{2}+s_{3}-s_{4}}(d).
\end{equation*}

\begin{remark}
It follows from multiplicativity that
\begin{equation*}
A_{q} \left(\frac{1}{2}, \frac{1}{2}, \frac{1}{2}, \frac{1}{2} \right) = \prod_{p \mid q} \left(1-\frac{1}{p} \right).
\end{equation*}
\end{remark}

\subsection{Applying Vorono\u{\i} Summation}
By Theorem~\ref{functional-equation-1}, $\mathcal{OD}_{2}^{\dagger}$ is equal to $\mathcal{P}$ plus
\footnotesize
\begin{align*}
& \frac{\psi(-1)}{\tau(\overline{\psi}) q} L(1+s_{1}-s_{2}, \overline{\psi}) 
\sum_{\ell = 1}^{\infty} \ell^{s_{2}-s_{1}-2} \ \sideset{}{^{\ast}} 
\sum_{h \tpmod{\ell q}} \int_{(\beta)} \mathring{g}(\tau, s_{2}) \zeta^{(\frac{h}{\ell q})}(\tau) 
\left(\frac{2\pi}{\ell q} \right)^{2s_{2}+s_{3}+s_{4}-2\tau-2} \Gamma(1+\tau-s_{2}-s_{3})\\
& \quad \times \Gamma(1+\tau-s_{2}-s_{4}) \left[\left\{e \left(\frac{s_{3}-s_{4}}{4} \right)+\psi(-1)e \left(-\frac{s_{3}-s_{4}}{4} \right) \right\} 
D_{\overline{\psi}} \left(1+\tau-s_{2}-s_{4}, s_{3}-s_{4}; \frac{\overline{h}}{\ell q} \right) \right.\\
& \quad - \left.\left\{e \left(\frac{2s_{2}+s_{3}+s_{4}-2\tau}{4} \right)+\psi(-1)e \left(-\frac{2s_{2}+s_{3}+s_{4}-2\tau}{4} \right) \right\} 
D_{\overline{\psi}} \left(1+\tau-s_{2}-s_{4}, s_{3}-s_{4}; -\frac{\overline{h}}{\ell q} \right) \right] \frac{d\tau}{2\pi i}.
\end{align*}
\normalsize
Executing the sum over $h$, we arrive at the following expression.
\begin{proposition}\label{proposition}
The off-diagonal contribution $\mathcal{OD}_{2}^{\dagger}$ extends meromorphically to $\mathcal{E}_{4, \beta}$, and there we have that
\begin{equation*}
\mathcal{OD}_{2}^{\dagger} = \mathcal{P}+\sum_{\pm} \mathcal{J}_{\pm},
\end{equation*}
where $\mathcal{P}$ is defined by~\eqref{P}, and
\begin{equation*}
\mathcal{J}_{\pm} \coloneqq \frac{\tau(\psi)}{q} \left(\frac{q}{2\pi} \right)^{1+s_{1}-s_{2}} L(1+s_{1}-s_{2}, \overline{\psi}) 
\sum_{m, n = 1}^{\infty} \frac{\sigma_{s_{3}-s_{4}}(n, \overline{\psi}) \mathcal{O}_{q}(m, \pm n; \Psi_{\texttt{s}}^{\pm})}
{m^{\frac{s_{1}+s_{2}+s_{3}+s_{4}-1}{2}} n^{\frac{1+s_{1}-s_{2}+s_{3}-s_{4}}{2}}}
\end{equation*}
with $\mathcal{O}_{q}(m, n; h)$ defined by~\eqref{O}, and
\begin{multline}\label{+}
\Psi_{\texttt{s}}^{+}(x) \coloneqq \left\{e \left(\frac{s_{3}-s_{4}}{4} \right)+\psi(-1)e \left(-\frac{s_{3}-s_{4}}{4} \right) \right\}\\
\times \int_{(\beta)} \left(\frac{x}{2} \right)^{s_{1}+s_{2}+s_{3}+s_{4}-1-2\tau} \Gamma(1+\tau-s_{2}-s_{3}) 
\Gamma(1+\tau-s_{2}-s_{4}) \mathring{g}(\tau, s_{2}) \frac{d\tau}{2\pi i},
\end{multline}
\begin{multline}\label{-}
\Psi_{\texttt{s}}^{-}(x) \coloneqq -\int_{(\beta)} \left\{e \left(\frac{2s_{2}+s_{3}+s_{4}-2\tau}{4} \right)
 + \psi(-1)e \left(-\frac{2s_{2}+s_{3}+s_{4}-2\tau}{4} \right) \right\}\\
\times \left(\frac{x}{2} \right)^{s_{1}+s_{2}+s_{3}+s_{4}-1-2\tau} \Gamma(1+\tau-s_{2}-s_{3}) 
\Gamma(1+\tau-s_{2}-s_{4}) \mathring{g}(\tau, s_{2}) \frac{d\tau}{2\pi i}.
\end{multline}
\end{proposition}

Proposition~\ref{proposition} reduces the problem to an analysis of sums of Kloosterman sums of moduli divisible by $q$. This is where the spectral theory of automorphic forms comes into play.

\subsection{Applying the Kuznetsov Formula}
We are in a position to apply Theorem~\ref{Kuznetsov-formula}~to~the sums of Kloosterman sums via a sophisticated analysis of integrals involving Bessel functions. Following Motohashi~\cite{Motohashi1997-2} verbatim, we first deal with the contribution of the plus sign. To reduce the integral transform $\mathscr{L}^{+}$ in light of the definition~\eqref{+}, we consider
\begin{equation}\label{double-integral}
\int_{0}^{\infty} J_{2it}(x) \int_{(\beta)} \Gamma(1+\tau-s_{2}-s_{3}) \Gamma(1+\tau-s_{2}-s_{4}) 
\left(\frac{x}{2} \right)^{s_{1}+s_{2}+s_{3}+s_{4}-2-2\tau} \mathring{g}(\tau, s_{2}) d\tau dx.
\end{equation}
In an appropriate domain as in~\cite[Equation (4.4.8)]{Motohashi1997-2}, this equals
\begin{equation}\label{four-gamma}
\int_{(\beta)} \frac{\Gamma(\frac{s_{1}+s_{2}+s_{3}+s_{4}-1}{2}-\tau+it)}
{\Gamma(\frac{3-s_{1}-s_{2}-s_{3}-s_{4}}{2}+\tau+it)} 
\Gamma(1+\tau-s_{2}-s_{3}) \Gamma(1+\tau-s_{2}-s_{4}) \mathring{g}(\tau, s_{2}) d\tau,
\end{equation}
because
\begin{equation*}
\int_{0}^{\infty} J_{\nu}(x) \left(\frac{x}{2} \right)^{-\mu} dx
 = \Gamma \left(\frac{1+\nu-\mu}{2} \right) \Gamma \left(\frac{1+\nu+\mu}{2} \right)^{-1}, \qquad \frac{1}{2} < \Re(\mu) < 1+\Re(\nu).
\end{equation*}
Since the integral~\eqref{four-gamma} is holomorphic in that domain, the analytic continuation implies that the double integral~\eqref{double-integral} is equal to~\eqref{four-gamma} throughout $\mathcal{E}_{4, \beta}$. Note that
\begin{equation}\label{Gamma}
\frac{\Gamma(a-\tau+it)}{\Gamma(1-a+\tau+it)}-\frac{\Gamma(a-\tau-it)}{\Gamma(1-a+\tau-it)}
 = \frac{2}{\pi i} \sinh(\pi t) \cos(\pi(a-\tau)) \Gamma(a-\tau+it) \Gamma(a-\tau-it),
\end{equation}
which follows from $\Gamma(s) \Gamma(1-s) = \frac{\pi}{\sin(\pi s)}$. Substituting $a = \frac{s_{1}+s_{2}+s_{3}+s_{4}-1}{2}$, we have after some rearrangements that
\begin{align*}
\mathscr{L}^{+} \Psi_{\texttt{s}}^{+}(t) &= \left\{e \left(\frac{s_{3}-s_{4}}{4} \right)+\psi(-1)e \left(-\frac{s_{3}-s_{4}}{4} \right) \right\} 
\int_{(\beta)} \sin \left(\frac{\pi(s_{1}+s_{2}+s_{3}+s_{4}-2\tau)}{2} \right)\\ 
& \quad \times \Gamma \left(\frac{s_{1}+s_{2}+s_{3}+s_{4}-1}{2}+it-\tau \right) 
\Gamma \left(\frac{s_{1}+s_{2}+s_{3}+s_{4}-1}{2}-it-\tau \right)\\ 
& \quad \times \Gamma(1+\tau-s_{2}-s_{3}) \Gamma(1+\tau-s_{2}-s_{4}) \mathring{g}(\tau, s_{2}) \frac{d\tau}{2\pi i}.
\end{align*}
Similarly, the holomorphic contribution boils down to
\begin{multline*}
\mathscr{L}^{\text{hol}} \Psi_{\texttt{s}}^{+}(k) =i^{k} \pi \left\{e \left(\frac{s_{3}-s_{4}}{4} \right)
 + \psi(-1)e \left(-\frac{s_{3}-s_{4}}{4} \right) \right\}\\
\times \int_{(\beta)} \frac{\Gamma(\frac{k-2+s_{1}+s_{2}+s_{3}+s_{4}}{2}-\tau)}
{\Gamma(\frac{k+2-s_{1}-s_{2}-s_{3}-s_{4}}{2}+\tau)} \Gamma(1+\tau-s_{2}-s_{3}) 
\Gamma(1+\tau-s_{2}-s_{4}) \mathring{g}(\tau, s_{2}) \frac{d\tau}{2\pi i}.
\end{multline*}

We now focus on the case of the minus sign for which Lemma~\ref{entire} plays an important r\^{o}le. By definition,
\begin{equation*}
\mathscr{L}^{-} \Psi_{\texttt{s}}^{-}(t) = 4 \cosh(\pi t) \int_{0}^{\infty} \Psi_{\texttt{s}}^{-}(x) K_{2it}(x) \frac{dx}{x}.
\end{equation*}
Inserting~\eqref{-}, we obtain an absolutely convergent integral, because
\begin{equation*}
K_{2it}(x) \ll 
	\begin{cases}
	|\log x| & \text{as $x \to 0$},\\
	\exp(-x) & \text{as $x \to \infty$}.
	\end{cases}
\end{equation*}
By Heaviside's integral formula
\begin{equation*}
\int_{0}^{\infty} K_{2v}(x) \left(\frac{x}{2} \right)^{2s-1} dx = \frac{\Gamma(s+v) \Gamma(s-v)}{2}, \qquad \Re(s) > |\Re(v)|,
\end{equation*}
there holds
\begin{align*}
\mathscr{L}^{-} \Psi_{\texttt{s}}^{-}(t) &= -\cosh(\pi t) \int_{(\beta)} 
\left\{e \left(\frac{2s_{2}+s_{3}+s_{4}-2\tau}{4} \right)+\psi(-1)e \left(-\frac{2s_{2}+s_{3}+s_{4}-2\tau}{4} \right) \right\}\\
& \quad \times \Gamma \left(\frac{s_{1}+s_{2}+s_{3}+s_{4}-1}{2}+it-\tau \right) 
\Gamma \left(\frac{s_{1}+s_{2}+s_{3}+s_{4}-1}{2}-it-\tau \right)\\
& \quad \times \Gamma(1+\tau-s_{2}-s_{3}) \Gamma(1+\tau-s_{2}-s_{4}) \mathring{g}(\tau, s_{2}) \frac{d\tau}{2\pi i}.
\end{align*}
We now define
\begin{align*}
\Phi_{\texttt{s}}^{+}(\xi) &\coloneqq -i(2\pi)^{s_{1}-s_{2}-2} \left\{e \left(\frac{s_{3}-s_{4}}{4} \right)+\psi(-1)e \left(-\frac{s_{3}-s_{4}}{4} \right) \right\}\\
& \quad \times \int_{-i\infty}^{i\infty} \sin \left(\frac{\pi(s_{1}+s_{2}+s_{3}+s_{4}-2\tau)}{2} \right)
\Gamma \left(\frac{s_{1}+s_{2}+s_{3}+s_{4}-1}{2}+\xi-\tau \right)\\
& \quad \times \Gamma \left(\frac{s_{1}+s_{2}+s_{3}+s_{4}-1}{2}-\xi-\tau \right) \Gamma(1+\tau-s_{2}-s_{3}) 
\Gamma(1+\tau-s_{2}-s_{4}) \mathring{g}(\tau, s_{2}) d\tau,\\
\Phi_{\texttt{s}}^{-}(\xi) &\coloneqq i(2\pi)^{s_{1}-s_{2}-2} \cos(\pi \xi) 
\int_{-i\infty}^{i\infty} \left\{e \left(\frac{2s_{2}+s_{3}+s_{4}-2\tau}{4} \right)+\psi(-1)e \left(-\frac{2s_{2}+s_{3}+s_{4}-2\tau}{4} \right) \right\}\\
& \quad \times \Gamma \left(\frac{s_{1}+s_{2}+s_{3}+s_{4}-1}{2}+\xi-\tau \right) 
\Gamma \left(\frac{s_{1}+s_{2}+s_{3}+s_{4}-1}{2}-\xi-\tau \right)\\
& \quad \times \Gamma(1+\tau-s_{2}-s_{3}) \Gamma(1+\tau-s_{2}-s_{4}) \mathring{g}(\tau, s_{2}) d\tau,\\
\Xi_{\texttt{s}}(\xi) &\coloneqq \int_{-i\infty}^{i\infty} \frac{\Gamma(\xi+\frac{s_{1}+s_{2}+s_{3}+s_{4}-1}{2}-\tau)}
{\Gamma(\xi+\frac{3-s_{1}-s_{2}-s_{3}-s_{4}}{2}+\tau)} 
\Gamma(1+\tau-s_{2}-s_{3}) \Gamma(1+\tau-s_{2}-s_{4}) \mathring{g}(\tau, s_{2}) \frac{d\tau}{2\pi i}.
\end{align*}
Here the contour in $\Phi_{\texttt{s}}^{+}(\xi)$ is curved to ensure that the poles of the first two gamma factors in the integrand lie to the right of the path and those of other factors lie to the left of the path; the variables $\xi, s_{1}, s_{2}, s_{3}, s_{4}$ are chosen such that the path can be drawn. The contour in $\Phi_{\texttt{s}}^{-}(\xi)$ is chosen in the same way. On the other hand, the contour in $\Xi$ separates the poles of $\Gamma(\xi+\frac{s_{1}+s_{2}+s_{3}+s_{4}-1}{2}-\tau)$ and those of $\Gamma(1+\tau-s_{2}-s_{3}) \Gamma(1+\tau-s_{2}-s_{4}) \mathring{g}(\tau, s_{2})$ to the left and the right of the path, respectively.
\begin{proposition}\label{three-functions}
Keep the notation as above. Then we have that
\begin{equation}\label{Phi+}
\Phi_{\texttt{s}}^{+}(\xi) = -\frac{(2\pi)^{s_{1}-s_{2}}}{4 \sin(\pi \xi)} 
\left\{e \left(\frac{s_{3}-s_{4}}{4} \right)+\psi(-1)e \left(-\frac{s_{3}-s_{4}}{4} \right) \right\} 
\left[\Xi_{\texttt{s}}(\xi)-\Xi_{\texttt{s}}(-\xi) \right],
\end{equation}
and
\begin{multline}\label{Phi-}
\Phi_{\texttt{s}}^{-}(\xi) = \frac{(2\pi)^{s_{1}-s_{2}}}{4i \sin(\pi \xi)}
\left[\left\{\psi(-1)e \left(\frac{s_{1}-s_{2}+2\xi}{4} \right)
 - e \left(-\frac{s_{1}-s_{2}+2\xi}{4} \right) \right\} \Xi_{\texttt{s}}(\xi) \right.\\
\left. - \left\{\psi(-1)e \left(\frac{s_{1}-s_{2}-2\xi}{4} \right)
 - e \left(-\frac{s_{1}-s_{2}-2\xi}{4} \right) \right\} \Xi_{\texttt{s}}(-\xi) \right],
\end{multline}
provided that the left-hand sides are well-defined. For $t \in \mathbb{R}$ and $\texttt{s} \in \mathcal{E}_{4, \beta}$, we have that
\begin{align*}
\mathscr{L}^{+} \Psi_{\texttt{s}}^{+}(t) = (2\pi)^{1-s_{1}+s_{2}} \Phi_{\texttt{s}}^{+}(it),\\
\mathscr{L}^{-} \Psi_{\texttt{s}}^{-}(t) = (2\pi)^{1-s_{1}+s_{2}} \Phi_{\texttt{s}}^{-}(it).
\end{align*}
For $k \in 2\mathbb{N}$ and $\texttt{s} \in \mathcal{E}_{4, \beta}$, we have that
\begin{equation*}
\mathscr{L}^{\mathrm{hol}} \Psi_{\texttt{s}}^{+}(k)
 = i^{k} \pi \left\{e \left(\frac{s_{3}-s_{4}}{4} \right)+\psi(-1)e \left(-\frac{s_{3}-s_{4}}{4} \right) \right\} 
\Xi_{\texttt{s}} \left(\frac{k-1}{2} \right).
\end{equation*}
\end{proposition}

\begin{proof}
We first establish~\eqref{Phi+}. It follows from~\eqref{Gamma} that
\begin{align*}
&\frac{\Gamma(\xi+\frac{s_{1}+s_{2}+s_{3}+s_{4}-1}{2}-\tau)}
{\Gamma(\xi+\frac{3-s_{1}-s_{2}-s_{3}-s_{4}}{2}+\tau)}
 - \frac{\Gamma(-\xi+\frac{s_{1}+s_{2}+s_{3}+s_{4}-1}{2}-\tau)}
{\Gamma(-\xi+\frac{3-s_{1}-s_{2}-s_{3}-s_{4}}{2}+\tau)}\\
& = -\frac{2}{\pi} \sin(\pi \xi) \sin \left(\frac{\pi(s_{1}+s_{2}+s_{3}+s_{4}-2\tau)}{2} \right)\\
& \quad \times \Gamma \left(\frac{s_{1}+s_{2}+s_{3}+s_{4}-1}{2}+\xi-\tau \right) 
\Gamma \left(\frac{s_{1}+s_{2}+s_{3}+s_{4}-1}{2}-\xi-\tau \right).
\end{align*}
In order to confirm~\eqref{Phi-}, we compute
\small
\begin{align*}
&\left\{\psi(-1)e \left(\frac{s_{1}-s_{2}+2\xi}{4} \right)-e \left(-\frac{s_{1}-s_{2}+2\xi}{4} \right) \right\} 
\frac{\Gamma(\xi+\frac{s_{1}+s_{2}+s_{3}+s_{4}-1}{2}-\tau)}
{\Gamma(\xi+\frac{3-s_{1}-s_{2}-s_{3}-s_{4}}{2}+\tau)}\\
& \quad - \left\{\psi(-1)e \left(\frac{s_{1}-s_{2}-2\xi}{4} \right)-e \left(-\frac{s_{1}-s_{2}-2\xi}{4} \right) \right\} 
\frac{\Gamma(-\xi+\frac{s_{1}+s_{2}+s_{3}+s_{4}-1}{2}-\tau)}
{\Gamma(-\xi+\frac{3-s_{1}-s_{2}-s_{3}-s_{4}}{2}+\tau)}\\
& = \frac{1}{\pi} \Gamma \left(\frac{s_{1}+s_{2}+s_{3}+s_{4}-1}{2}+\xi-\tau \right) 
\Gamma \left(\frac{s_{1}+s_{2}+s_{3}+s_{4}-1}{2}-\xi-\tau \right)\\
& \quad \times \left[\left\{\psi(-1)e \left(\frac{s_{1}-s_{2}+2\xi}{4} \right)-e \left(-\frac{s_{1}-s_{2}+2\xi}{4} \right) \right\} 
\sin \left(\pi \left(\frac{s_{1}+s_{2}+s_{3}+s_{4}-1}{2}-\xi-\tau \right) \right) \right.\\
&\left. \quad - \left\{\psi(-1)e \left(\frac{s_{1}-s_{2}-2\xi}{4} \right)-e \left(-\frac{s_{1}-s_{2}-2\xi}{4} \right) \right\} 
\sin \left(\pi \left(\frac{s_{1}+s_{2}+s_{3}+s_{4}-1}{2}+\xi-\tau \right) \right) \right].
\end{align*}
\normalsize
Via an elementary manipulation of trigonometrics, the right-hand side is equal to
\begin{multline*}
\frac{\sin(2\pi \xi)}{\pi i} \left\{e \left(\frac{2s_{2}+s_{3}+s_{4}-2\tau}{4} \right)+\psi(-1)e \left(-\frac{2s_{2}+s_{3}+s_{4}-2\tau}{4} \right) \right\}\\
\times \Gamma \left(\frac{s_{1}+s_{2}+s_{3}+s_{4}-1}{2}+\xi-\tau \right) 
\Gamma \left(\frac{s_{1}+s_{2}+s_{3}+s_{4}-1}{2}-\xi-\tau \right)
\end{multline*}
Proposition~\ref{three-functions} now follows from the relation $\frac{\sin(2\pi \xi)}{\sin(\pi \xi)} = 2 \cos(\pi \xi)$.
\end{proof}

\subsection{Cubic Moment}
The absolute convergence is guaranteed when the double summation over $m$ and $n$ is concerned, because we know the trivial bound for the Hecke eigenvalues and $\texttt{s} \in \mathcal{E}_{4, \beta}$. The problem boils down to bounding $\mathscr{L}^{\pm} \Psi_{\texttt{s}}^{\pm}$, and Proposition~\ref{three-functions} reduces our task to the analysis of $\Xi$. If $t \in \mathbb{R}$ and $k \in 2\mathbb{N}$ tend to infinity, then
\begin{equation*}
\Xi_{\texttt{s}}(it) \ll |t|^{-A}, \qquad \Xi_{\texttt{s}} \left(\frac{k-1}{2} \right) \ll k^{-A}
\end{equation*}
uniformly for bounded $\texttt{s}$. Hence we deduce the meromorphic continuation of $\mathcal{J}_{\pm}$ to $\mathbb{C}^{4}$. This shows that the decomposition~\eqref{Atkinson-dissection} is valid throughout $\mathbb{C}^{4}$.

The contribution of Hecke--Maa{\ss} newforms is calculated as
\begin{multline*}
\sum_{n = 1}^{\infty} \frac{\sigma_{s_{3}-s_{4}}(n, \overline{\psi}) 
\overline{\rho_{f, M, q}(\pm n)}}{n^{\frac{1+s_{1}-s_{2}+s_{3}-s_{4}}{2}}}
 = \frac{\varepsilon_{f}^{\frac{1 \mp 1}{2}}}{L(1, \operatorname{ad} f)^{\frac{1}{2}}(Mq \nu(q))^{\frac{1}{2}}} \prod_{p \mid q_{0}} \left(1-\frac{1}{p^{2}} \right)^{\frac{1}{2}}\\
\times \sum_{d \mid M} \xi_{f}(M, d) d^{\frac{1-s_{1}+s_{2}-s_{3}+s_{4}}{2}} \sum_{n = 1}^{\infty} 
\frac{\sigma_{s_{3}-s_{4}}(n, \overline{\psi}) \lambda_{f}(n)}{n^{\frac{1+s_{1}-s_{2}+s_{3}-s_{4}}{2}}},
\end{multline*}
where $f \in \mathcal{B}^{\ast}(q_{0})$ with $q_{0} M \mid q$. We here used the property $\sigma_{s_{3}-s_{4}}(nd, \overline{\psi}) = \sigma_{s_{3}-s_{4}}(n, \overline{\psi})$ since $d$ is not coprime to $q$. Note that the Hecke relation implies
\begin{equation*}
\sum_{n = 1}^{\infty} \frac{\sigma_{s_{3}-s_{4}}(n, \overline{\psi}) \lambda_{f}(n)}{n^{\frac{1+s_{1}-s_{2}+s_{3}-s_{4}}{2}}}
 = \frac{L(\frac{1+s_{1}-s_{2}+s_{3}-s_{4}}{2}, f) 
L(\frac{1+s_{1}-s_{2}-s_{3}+s_{4}}{2}, f \otimes \overline{\psi})}{L(1+s_{1}-s_{2}, \overline{\psi})}.
\end{equation*}
The denominator now cancels out with $L(1+s_{1}-s_{2}, \overline{\psi})$ appearing in $\mathcal{J}_{\pm}$ in Proposition~\ref{proposition}. Evaluating the sum over $m$ in a similar fashion, we obtain
\begin{multline*}
\sum_{m = 1}^{\infty} \frac{\rho_{f, M, q}(m)}{m^{\frac{s_{1}+s_{2}+s_{3}+s_{4}-1}{2}}}
 = \frac{1}{L(1, \operatorname{ad} f)^{\frac{1}{2}}(Mq \nu(q))^{\frac{1}{2}}} \prod_{p \mid q_{0}} \left(1-\frac{1}{p^{2}} \right)^{\frac{1}{2}}\\
\times \sum_{d \mid M} \xi_{f}(M, d) d^{\frac{3-s_{1}-s_{2}-s_{3}-s_{4}}{2}} L \left(\frac{s_{1}+s_{2}+s_{3}+s_{4}-1}{2}, f \right).
\end{multline*}
The central $L$-values $L(\frac{1}{2}, f)^{2} L(\frac{1}{2}, f \otimes \overline{\psi})$ emerge if we take $\texttt{s} \mapsto (\frac{1}{2}, \frac{1}{2}, \frac{1}{2}, \frac{1}{2})$. A similar process applies to the Eisenstein contribution, namely
\begin{multline*}
\sum_{n = 1}^{\infty} \frac{\sigma_{s_{3}-s_{4}}(n, \overline{\psi}) 
\overline{\rho_{\chi, M, q}(n, t)}}{n^{\frac{1+s_{1}-s_{2}+s_{3}-s_{4}}{2}}}
 = \frac{\overline{C(\chi, M, t)}}{(q \nu(q))^{\frac{1}{2}} \widetilde{\mathfrak{n}}_{q}(M) L^{q}(1-2it, \overline{\chi}^{2})} 
\left(\frac{M_{1}}{M_{2}} \right)^{\frac{1}{2}} \sum_{\delta \mid M_{2}} \delta \mu \left(\frac{M_{2}}{\delta} \right) \chi(\delta)\\
\times \sum_{n = 1}^{\infty} \frac{\sigma_{s_{3}-s_{4}}(n, \overline{\psi})}{n^{\frac{1+s_{1}-s_{2}+s_{3}-s_{4}}{2}+it}} 
\sum_{\substack{cM_{1} \delta f = n \\ (c, \frac{q}{M}) = 1}} \overline{\chi}(c) \chi(f) c^{2it}.
\end{multline*}
The second line of the above equality boils down to
\begin{equation}\label{double-sum}
(M_{1} \delta)^{-\frac{1+s_{1}-s_{2}+s_{3}-s_{4}}{2}-it} \sum_{(c, \frac{q}{M}) = 1} \sum_{f = 1}^{\infty} \frac{\overline{\chi}(c) \chi(f) 
\sigma_{s_{3}-s_{4}}(cM_{1} \delta f, \overline{\psi})}{c^{\frac{1+s_{1}-s_{2}+s_{3}-s_{4}}{2}-it} f^{\frac{1+s_{1}-s_{2}+s_{3}-s_{4}}{2}+it}},
\end{equation}
and Lemma~\ref{Hecke-multiplicativity} yields
\begin{multline*}
\sigma_{s_{3}-s_{4}}(cM_{1} \delta f, \overline{\psi}) 
 = \mathop{\sum \sum \sum}_{e_{1} \mid c, \, e_{1} e_{2} \mid f, \, e_{2} \mid \delta} \mu(e_{1}) \mu(e_{2}) 
\overline{\psi}(e_{1} e_{2}) (e_{1} e_{2})^{s_{3}-s_{4}}\\ \times \sigma_{s_{3}-s_{4}} \left(\frac{c}{e_{1}}, \overline{\psi} \right) 
\sigma_{s_{3}-s_{4}} \left(\frac{f}{e_{1} e_{2}}, \overline{\psi} \right) \sigma_{s_{3}-s_{4}} \left(\frac{\delta}{e_{2}}, \overline{\psi} \right),
\end{multline*}
where $\sigma_{s_{3}-s_{4}}(M_{1}, \overline{\psi}) = 1$. In anticipation of the ensuing simplifications, we make a change of variables $c \mapsto ce_{1}$ and $f \mapsto fe_{1} e_{2}$. Then the double sum in~\eqref{double-sum} becomes
\begin{equation*}
\sum_{(c, \frac{q}{M}) = 1} \sum_{f = 1}^{\infty} \frac{\overline{\chi}(c) \chi(f) \sigma_{s_{3}-s_{4}}(c, \overline{\psi}) 
\sigma_{s_{3}-s_{4}}(f, \overline{\psi})}{c^{\frac{1+s_{1}-s_{2}+s_{3}-s_{4}}{2}-it} f^{\frac{1+s_{1}-s_{2}+s_{3}-s_{4}}{2}+it}} 
\sum_{e_{1} = 1}^{\infty} \sum_{e_{2} \mid \delta} \frac{\mu(e_{1}) \mu(e_{2}) \chi(e_{2}) \overline{\psi}(e_{1} e_{2}) 
\sigma_{s_{3}-s_{4}}(\frac{\delta}{e_{2}}, \overline{\psi})}{e_{1}^{1+s_{1}-s_{2}} e_{2}^{\frac{1+s_{1}-s_{2}-s_{3}+s_{4}}{2}+it}}.
\end{equation*}
Consequently, the sums over $c$, $e_{1}$, $f$ are equal to
\begin{equation*}
\frac{L(\frac{1+s_{1}-s_{2}+s_{3}-s_{4}}{2}+it, \chi) 
L(\frac{1+s_{1}-s_{2}-s_{3}+s_{4}}{2}+it, \chi \overline{\psi}) 
L^{\frac{q}{M}}(\frac{1+s_{1}-s_{2}+s_{3}-s_{4}}{2}-it, \overline{\chi}) 
L(\frac{1+s_{1}-s_{2}-s_{3}+s_{4}}{2}-it, \overline{\chi \psi})}{L(1+s_{1}-s_{2}, \overline{\psi})}.
\end{equation*}
On the other hand, the sum over $m$ is calculated as
\begin{multline*}
\sum_{m = 1}^{\infty} \frac{\rho_{\chi, M, q}(m)}{m^{\frac{s_{1}+s_{2}+s_{3}+s_{4}-1}{2}}}
 =  \frac{C(\chi, M, t)}{(q \nu(q))^{\frac{1}{2}} \widetilde{\mathfrak{n}}_{q}(M) L^{q}(1+2it, \chi^{2})} \left(\frac{M_{1}}{M_{2}} \right)^{\frac{1}{2}} 
\sum_{\delta \mid M_{2}} \delta \mu \left(\frac{M_{2}}{\delta} \right) \overline{\chi}(\delta)\\
\times (M_{1} \delta)^{\frac{1-s_{1}-s_{2}-s_{3}-s_{4}}{2}+it} \sum_{(c, \frac{q}{M}) = 1} \sum_{f = 1}^{\infty} 
\frac{\chi(c) \overline{\chi}(f)}{c^{\frac{s_{1}+s_{2}+s_{3}+s_{4}-1}{2}+it} f^{\frac{s_{1}+s_{2}+s_{3}+s_{4}-1}{2}-it}}.
\end{multline*}
The second line of the above equality boils down to
\begin{equation*}
(M_{1} \delta)^{\frac{1-s_{1}-s_{2}-s_{3}-s_{4}}{2}+it} L^{\frac{q}{M}} \left(\frac{s_{1}+s_{2}+s_{3}+s_{4}-1}{2}+it, \chi \right) 
L \left(\frac{s_{1}+s_{2}+s_{3}+s_{4}-1}{2}-it, \overline{\chi} \right).
\end{equation*}
The contribution of holomorphic newforms follows the above treatment verbatim.

Hence, the resulting form of the cubic moment is
\begin{equation*}
\mathcal{J}_{\pm} \coloneqq \mathcal{J}_{\pm}^{\text{Maa{\ss}}}
 + \mathcal{J}_{\pm}^{\text{Eis}}+\delta_{\pm = +} \mathcal{J}_{+}^{\text{hol}},
\end{equation*}
where
\begin{multline*}
\mathcal{J}_{\pm}^{\text{Maa{\ss}}} \coloneqq \sum_{q_{0} M \mid q} 
\sum_{f \in \mathcal{B}^{\ast}(q_{0})} \frac{L(\frac{s_{1}+s_{2}+s_{3}+s_{4}-1}{2}, f) 
L(\frac{1+s_{1}-s_{2}+s_{3}-s_{4}}{2}, f) L(\frac{1+s_{1}-s_{2}-s_{3}+s_{4}}{2}, f \otimes \overline{\psi})}
{L(1, \operatorname{ad} f)}\\
\times \Theta_{q}^{\text{Maa{\ss}}}(s_{1}, s_{2}, s_{3}, s_{4}, f),
\end{multline*}
\begin{multline*}
\mathcal{J}_{\pm}^{\text{Eis}} \coloneqq \sum_{c_{\chi}^{2} \mid M \mid q} 
\int_{\mathbb{R}} \frac{L(\frac{1+s_{1}-s_{2}+s_{3}-s_{4}}{2}+it, \chi) 
L^{\frac{q}{M}}(\frac{s_{1}+s_{2}+s_{3}+s_{4}-1}{2}+it, \chi) 
L(\frac{1+s_{1}-s_{2}-s_{3}+s_{4}}{2}+it, \chi \overline{\psi})}{L^{q}(1+2it, \chi^{2})}\\
\times \frac{L(\frac{s_{1}+s_{2}+s_{3}+s_{4}-1}{2}-it, \overline{\chi}) 
L^{\frac{q}{M}}(\frac{1+s_{1}-s_{2}+s_{3}-s_{4}}{2}-it, \overline{\chi}) 
L(\frac{1+s_{1}-s_{2}-s_{3}+s_{4}}{2}-it, \overline{\chi \psi})}{L^{q}(1-2it, \overline{\chi}^{2})}\\
\times \Theta_{\chi, q}^{\text{Eis}}(s_{1}, s_{2}, s_{3}, s_{4}, t) \frac{dt}{2\pi},
\end{multline*}
\begin{multline*}
\mathcal{J}_{+}^{\text{hol}} \coloneqq  \sum_{q_{0} M \mid q} 
\sum_{f \in \mathcal{B}_{\text{hol}}^{\ast}(q_{0})} \frac{L(\frac{s_{1}+s_{2}+s_{3}+s_{4}-1}{2}, f) 
L(\frac{1+s_{1}-s_{2}+s_{3}-s_{4}}{2}, f) L(\frac{1+s_{1}-s_{2}-s_{3}+s_{4}}{2}, f \otimes \overline{\psi})}
{L(1, \operatorname{ad} f)}\\
\times \Theta_{q}^{\text{hol}}(s_{1}, s_{2}, s_{3}, s_{4}, f)
\end{multline*}
with the local factors defined by
\begin{multline}\label{Theta-1}
\Theta_{q}^{\text{Maa{\ss}}}(s_{1}, s_{2}, s_{3}, s_{4}, f)
\coloneqq \varepsilon_{f}^{\frac{1 \mp 1}{2}} \frac{q^{s_{1}-s_{2}-1} \tau(\psi)}{M \nu(q)} \prod_{p \mid q_{0}} 
\left(1-\frac{1}{p^{2}} \right)\\
\times \sum_{d_{1} \mid M} \sum_{d_{2} \mid M} \frac{\xi_{f}(M, d_{1}) \xi_{f}(M, d_{2}) 
\Phi_{\texttt{s}}^{\pm}(it_{f})}{d_{1}^{\frac{s_{1}+s_{2}+s_{3}+s_{4}-3}{2}} d_{2}^{\frac{s_{1}-s_{2}+s_{3}-s_{4}-1}{2}}},
\end{multline}
\begin{multline}\label{Theta-2}
\Theta_{\chi, q}^{\text{Eis}}(s_{1}, s_{2}, s_{3}, s_{4}, t)
\coloneqq \frac{q^{s_{1}-s_{2}-1} \tau(\psi)}{\nu(q) \widetilde{\mathfrak{n}}_{q}(M)^{2}} \frac{M_{1}^{1-s_{1}-s_{3}}}{M_{2}} 
\sum_{\delta_{1} \mid M_{2}} \sum_{\delta_{2} \mid M_{2}} \frac{\mu(\frac{M_{2}}{\delta_{1}}) 
\mu(\frac{M_{2}}{\delta_{2}}) \overline{\chi}(\delta_{1}) \chi(\delta_{2})}{\delta_{1}^{\frac{s_{1}+s_{2}+s_{3}+s_{4}-3}{2}-it} 
\delta_{2}^{\frac{s_{1}-s_{2}+s_{3}-s_{4}-1}{2}+it}}\\
\times \sum_{e \mid \delta_{2}} \mu(e) \chi(e) \overline{\psi}(e) 
e^{-\frac{1+s_{1}-s_{2}-s_{3}+s_{4}}{2}-it} \sigma_{s_{3}-s_{4}} \left(\frac{\delta}{e}, \overline{\psi} \right) \Phi_{\texttt{s}}^{\pm}(it),
\end{multline}
\begin{multline}\label{Theta-3}
\Theta_{q}^{\text{hol}}(s_{1}, s_{2}, s_{3}, s_{4}, f) \coloneqq \frac{i^{k} \tau(\psi)}{2Mq \nu(q)} 
\left(\frac{q}{2\pi} \right)^{s_{1}-s_{2}} \left\{e \left(\frac{s_{3}-s_{4}}{4} \right)+\psi(-1)e \left(-\frac{s_{3}-s_{4}}{4} \right) \right\}\\
\times \prod_{p \mid q_{0}} \left(1-\frac{1}{p^{2}} \right) \sum_{d_{1} \mid M} 
\sum_{d_{2} \mid M} \frac{\xi_{f}(M, d_{1}) \xi_{f}(M, d_{2})}
{d_{1}^{\frac{s_{1}+s_{2}+s_{3}+s_{4}-3}{2}} d_{2}^{\frac{s_{1}-s_{2}+s_{3}-s_{4}-1}{2}}} \Xi_{\texttt{s}} \left(\frac{k_{f}-1}{2} \right).
\end{multline}
The proof of Theorem~\ref{main} is now complete.\qed

\section{Proof of Corollary~\ref{corollary}}
Given a parameter $H$ at our disposal, we specialise the test function in Theorem~\ref{main} as
\begin{equation*}
g(t) = \frac{1}{\sqrt{\pi} H} \exp \left(-\left(\frac{t-T}{H} \right)^{2} \right).
\end{equation*}
Suppose that
\begin{equation*}
\sqrt{T} \leq H \leq T(\log T)^{-1}.
\end{equation*}
The corresponding integral transform $\Xi$ is analysed in~\cite[Equations (5.1.40)--(5.1.42)]{Motohashi1997-2} that are also useful in our context. We derive an asymptotic formula that is almost identical to~\cite[Theorem 5.1]{Motohashi1997-2}. Estimating the spectral sum by absolute values yields
\begin{multline}\label{inequality}
\int_{T}^{T+H} \left|\zeta \left(\frac{1}{2}+it \right) L \left(\frac{1}{2}+it, \psi \right) \right|^{2} dt\\
\ll_{\varepsilon} H^{1+\varepsilon} q^{\varepsilon}+\frac{H^{\frac{3}{2}} q^{-\frac{1}{2}+\varepsilon}}{T} \sum_{t_{f} \ll \frac{T}{H}} \left|L \left(\frac{1}{2}, f \right)^{2} L \left(\frac{1}{2}, f \otimes \psi \right) \right|,
\end{multline}
where the level of Hecke--Maa{\ss} newforms is $q$. Here we also evaluated the finite sums in the local factors using~\eqref{xi}. The truncation of the spectral sum is justified since the Taylor~series expansion of $\exp(-(Ht_{f}/T)^{2}/4)$ in~\cite[Equation (5.1.44)]{Motohashi1997-2} implies that the rest of the spectral sum is smaller than $O_{\varepsilon}(H^{1+\varepsilon} q^{\varepsilon})$. Because $\psi$ is quadratic, the central $L$-values in~\eqref{inequality} are nonnegative, and one can apply H\"{o}lder's inequality with exponent $(\frac{3}{2}, 3)$. It follows from the result of Conrey--Iwaniec~\cite{ConreyIwaniec2000} or Petrow--Young~\cite{PetrowYoung2023} that
\begin{equation*}
H^{1+\varepsilon} q^{\varepsilon}+\frac{H^{\frac{3}{2}}}{T} q^{-\frac{1}{2}+\varepsilon} \left(\frac{qT^{2}}{H^{2}} \right)^{\frac{2}{3}+\varepsilon} 
\left(\frac{qT^{2}}{H^{2}} \right)^{\frac{1}{3}+\varepsilon} \ll_{\varepsilon} H^{1+\varepsilon} q^{\varepsilon}+\left(\frac{qT^{2}}{H} \right)^{\frac{1}{2}+\varepsilon}.
\end{equation*}
Optimising $H = q^{\frac{1}{3}} T^{\frac{2}{3}}$ completes the proof of Corollary~\ref{corollary}.


\providecommand{\bysame}{\leavevmode\hbox to3em{\hrulefill}\thinspace}
\providecommand{\MR}{\relax\ifhmode\unskip\space\fi MR }
\providecommand{\MRhref}[2]{%
  \href{http://www.ams.org/mathscinet-getitem?mr=#1}{#2}
}
\providecommand{\href}[2]{#2}

\end{document}